\documentclass{amsart}
\usepackage{amsfonts,amsmath,amssymb,amscd,amstext,amsbsy,latexsym}
\newtheorem{thm}{Theorem}[section]
\newtheorem{cor}[thm]{Corollary}
\newtheorem{lem}[thm]{Lemma}
\newtheorem{prop}[thm]{Proposition}
\theoremstyle{definition}
\newtheorem{defn}[thm]{Definition}
\theoremstyle{remark}
\newtheorem{rem}[thm]{Remark}
\theoremstyle{Example}

\numberwithin{equation}{section}

\newcommand{\A}{\mathcal{A}}
\begin{document}

\title{Submaximal Integral Domains}%
\author{A. Azarang}%
\maketitle
\centerline{Department of Mathematics, Chamran University, Ahvaz,
Iran} \centerline{a${}_{-}$azarang@scu.ac.ir}

\begin{abstract}
It is proved that if $D$ is a $UFD$ and $R$ is a $D$-algebra,
such that $U(R)\cap D\neq U(D)$, then $R$ has a maximal subring.
In particular, if $R$ is a ring which either contains a unit $x$
which is not algebraic over the prime subring of $R$, or $R$ has
zero characteristic and there exists a natural number $n>1$ such
that $\frac{1}{n}\in R$, then $R$ has a maximal subring. It is
shown that if $R$ is a reduced ring with $|R|>2^{2^{\aleph_0}}$
or $J(R)\neq 0$, then any $R$-algebra has a maximal subring.
Residually finite rings without maximal subrings are fully
characterized. It is observed that every uncountable $UFD$ has a
maximal subring. The existence of maximal subrings in a
noetherian integral domain $R$, in relation to either the
cardinality of the set of divisors of some of its elements or the
height of its maximal ideals, is also investigated.
\end{abstract}

\section*{Introduction}

All rings in this article are commutative with $1\neq 0$; all
modules are unital. If $S$ is a subring of a ring $R$, then
$1_R\in S$. In this paper the characteristic of a ring $R$ is
denoted by $Char(R)$, and the set of all maximal ideals of a ring
$R$ is denoted by $Max(R)$. For any ring $R$, let
$Z=\mathbb{Z}\cdot 1_R=\{n\cdot 1_R\ |\ n\in \mathbb{Z} \}$, be
the prime subring of $R$. Rings with maximal subrings are called
submaximal rings in \cite{azkrm} and \cite{azkarm4}. Some
important rings such as uncountable artinian rings,
zero-dimensional rings which are either not integral over $Z$ or
with zero characteristic, noetherian rings $R$ with
$|R|>2^{\aleph_0}$ and infinite direct product of rings are
submaximal, see [4-7]. We should remind the reader that all finite
rings except $\mathbb{Z}_n$ (up to isomorphism), where $n$ is a
natural number, are submaximal. It is also interesting to note
that whenever $S$ is a finite maximal subring of a ring $R$, then
$R$ must be finite, see \cite[Theorem 8]{bl1}, \cite{laffey},
\cite{klein} and \cite{lee}. The latter interesting fact is also
an easy consequence of \cite[the proof of Theorem 2.9 ]{azkrm2}
or \cite[Theorem 3.8]{azkarm3}. Recently S.S. Korobkov determined
which finite rings have exactly two maximal subrings, see
\cite{kor}.\\

We remind the reader that whenever $S$ is a maximal subring of a
ring $R$, then $R$ is called minimal ring extension of $S$.
Recently, D.E. Dobbs and J. Shapiro have extended the results in
\cite{frd}, to integral domains and certain commutative rings, see
\cite{db8} and \cite{db9}, respectively. Also see \cite{abmin},
\cite{cahen2} and \cite{lucas}. T.G. Lucas, in \cite{lucas},
characterized minimal ring extensions of certain commutative rings
especially in the case of minimal integral extension. It is
interesting to know that every commutative ring $R$ has a minimal
ring extension, for if $M$ is a simple $R$-module then the
idealization $R(+)M$ is a minimal ring extension of $R$ (note, for
any $R$-module $M$, every $R$-subalgebra of $R(+)M$ has the form
$R(+)N$, where $N$ is a submodule of $M$, see \cite{db4}). For a
generalization of minimal ring extensions, see also
\cite{cahen1}.\\

Unlike maximal ideals (resp. minimal ring extension) whose
existence is guaranteed either by Zorn Lemma or noetherianity of
rings (resp. by idealization or other techniques, see \cite{db4}),
maximal subrings need not always exist, see \cite{azkarm4} for
such examples and in particular, for example of rings of any
infinite cardinality, which are not submaximal. In fact by the
above comment about the idealization, one can easily see that if
$K$ is any field with zero characteristic, then the ring
$\mathbb{Z}(+)K$ is not submaximal, see \cite[Example
3.19]{azkarm4}. Also, in the latter reference and in \cite{azkrm}
a good motivations for the study of maximal subrings related to
algebraic geometry and elliptic curves are given.\\

In this paper, we are merely interested in finding submaximal
integral domains, especially atomic and noetherian integral
domains. A brief outline of this paper is as follow. Section 1,
contains some preliminaries and also some generalizations of
results which are to be appeared in \cite{azkarm4}. It is
observed that whenever $D$ is a $UFD$ and $R$ is a $D$-algebra in
which a non-unit of $D$ is invertible, then $R$ is submaximal. In
particular, if $R$ is a $\mathbb{Z}$-algebra such that
$\mathbb{Q}\cap R\neq \mathbb{Z}$, then $R$ is submaximal.
Moreover, if $D$ is a $PID$ and $D\subseteq R$ is an integral
domain such that $D$ is integrally closed in $R$ and $U(R)\neq
U(D)$, then $R$ is submaximal. Consequently it is proved that, if
$R$ is a $\mathbb{Z}$-algebra, then either $R$ is submaximal or
for any prime number $p$, there exists a maximal ideal $M$ of $R$
such that $Char(\frac{R}{M})=p$. It is observed that every ring
either is submaximal or is Hilbert. In particular, if $R$ is a
reduced ring with $|R|>2^{2^{\aleph_0}}$ or $J(R)\neq 0$, then
any $R$-algebra is submaximal. Consequently, it is shown that if
$R$ is a reduced non-submaximal ring with zero characteristic,
then $\bigcap_{p\in\mathbb{P}}Rp=0$. It is proved that if $R$ is a
residue finite non-submaximal ring, then $R$ is a countable
principal ideal ring which is either an integral domain with zero
characteristic or it is an artinian ring with nonzero
characteristic. Finally in Section 1, the existence of maximal
subring in semi-local rings and localization of rings are
investigated. In particular, it is proved that if $R$ is a ring
and $S$ is a multiplicatively closed set in $R$ such that $R_S$
is semi-local, then either $R_S$ is submaximal or every prime
ideal of $R_S$ has the form $P_S$, for some $P\in Max(R)\cap
Min(R)$. Moreover, in the latter case, if $R_S$ is submaximal,
then $R$ is submaximal too. Section 2, is devoted to the
existence of maximal subrings in unique factorization domains,
noetherian integral domains and certain atomic domains. It is
observed that, every uncountable $UFD$ is submaximal. We also
generalized the latter result to certain uncountable atomic
domains. In particular, it is proved that if $R$ is an
uncountable noetherian $\mathbb{Z}$-algebra, in which every
natural number has at most countably many (irreducible) divisors,
then $R$ is submaximal. It is shown that, if $R$ is a noetherian
integral domain with zero characteristic and $tr.deg_Z R=n\geq 1$
(resp. with nonzero characteristic and $tr.deg_Z R=n\geq 2$) such
that the height of every maximal ideal of $R$ is greater or equal
to $n+1$ (resp. greater or equal to $n$) and $Z[X]\subseteq R$ is
a residually algebraic extension, where $X$ is a transcendental
basis for $R$ over $Z$, then $R$ is submaximal. Finally, we show
that every uncountable Dedekind
domain $D$ with $|Max(D)|\leq \aleph_0$, is submaximal.\\

Next, let us recall some standard definitions and notations in
commutative rings, see \cite{kap}, which are used in the sequel.
An integral domain $D$ is called $G$-domain if the quotient field
of $D$ is finitely generated as a ring over $D$. A prime ideal $P$
of a ring $R$ is called $G$-ideal if $\frac{R}{P}$ is a
$G$-domain. A ring $R$ is called Hilbert if every $G$-ideal of
$R$ is maximal. We also call a ring $R$, not necessarily
noetherian, semi-local (resp. local) if $Max(R)$ is finite (resp.
$|Max(R)|=1$). An integral domain $D$ is called atomic, if every
nonzero non-unit of $D$ is a finite product of irreducible
elements, not necessarily unique. An integral domain $D$ is
called $idf$-domain if every nonzero non-unit element of $D$ has
at most finitely many irreducible divisors, see \cite{anderson}.
In this paper the set of minimal prime ideals and prime ideals of
a ring $R$ are denoted by $Min(R)$ and $Spec(R)$, respectively. As
usual, let $U(R)$ denote the set of all units of a ring $R$. The
Jacobson and the nil radical ideals of a ring $R$ are also
denoted by $J(R)$ and $N(R)$, respectively.  If $P$ is a prime
ideal of a ring $R$, then the height of $P$ is denoted by
$ht(P)$. If $D$ is an integral domain, then we denote the set of
all non-associate irreducible elements of $D$ by $Ir(D)$. We
recall that if $D\subseteq R$ is an extension of integral
domains, then as for the existence a transcendental basis for
field extensions, one can easily see that there exists a subset
$X$ of $R$ which is algebraically independent over $D$ and $R$ is
algebraic over $D[X]$ (hence every integral domain is algebraic
over a $UFD$). Moreover, in the latter case $|X|=tr.deg_F(E)$,
where $E$ and $F$ are the quotient fields of $R$ and $D$,
respectively. Hence, similar to the field extensions, we can
define the transcendental degree of $R$ over $D$ which is denoted
by $tr.deg_D(R)$. Finally, we denote the set of all natural prime
numbers by $\mathbb{P}$.

\section{Preliminaries and Generalizations}
We begin this section with the following useful fact about the
existence of maximal subrings in subrings of a submaximal ring,
which is the converse of \cite[Proposition 2.1]{azkarm3}. We
remind the reader that a ring $R$ is submaximal if and only if
there exist a proper subring $S$ of $R$ and an element $x\in
R\setminus S$ such that $S[x]=R$, see \cite[Theorem
2.5]{azarang}. Now the following is in order, and although its
proof is in \cite{azkarm4}, but we present it for the sake of the
reader.

\begin{prop}\label{pp1}
\cite[Theorem 2.19]{azkarm4}. Let $R\subseteq T$ be rings. If
there exists a maximal subring $V$ for $T$ such that $V$ is
integrally closed in $T$ and $U(R)\nsubseteq V$, then $R$ is
submaximal.
\end{prop}
\begin{proof}
First, we claim that whenever $x\in U(R)\setminus V$, then
$x^{-1}\in V$. To see this, we observe that $x^{-1}\in R\subseteq
T=V[x]$ (note, $V$ is a maximal subring of $T$). Consequently,
$x^{-1}=a_0+a_1x+\cdots+a_nx^n$, where $a_0,a_1,\ldots,a_n\in V$.
Now by multiplying the latter equality by $x^{-n}$, we infer that
$x^{-1}$ is integral over $V$, hence $x^{-1}\in V$. But
$U(R)\nsubseteq V$ implies that $V\cap R$ is a proper subring of
$R$ and there exists $x\in U(R)\setminus V$ with $T=V[x]$.
Finally, we claim that $R=(V\cap R)[x]$, which by the preceding
comment, it implies that $R$ is submaximal. To this end, let
$y\in R$, hence $y\in V[x]$ and therefore
$y=b_0+b_1x+\cdots+b_mx^m$, where $b_0,b_1,\ldots,b_m\in V$,
implies that $yx^{-m}\in V\cap R$ (note, $x^{-1}\in V$), i.e.,
$y\in (R\cap V)[x]$ and we are done.
\end{proof}

Next, we have the following fact which is needed in the sequel.

\begin{thm}\label{pp2}
Let $R$ be a ring and $D$ be a subring of $R$ which is a $UFD$.
If there exists an irreducible element $p\in D$ such that
$\frac{1}{p}\in R$, then $R$ is submaximal. In particular, if
$U(R)\cap D\neq U(D)$, then $R$ is submaximal.
\end{thm}
\begin{proof}
We first prove the theorem by the assumption that $R$ is algebraic
over $D$. We also may assume that $R$ is an integral domain
(note, if not, then there exists a prime ideal $Q$ of $R$ such
that $D\cap Q=0$ and therefore $\frac{R}{Q}$ contains a copy of
$D$). Now suppose that $K$ and $E$ are the quotient fields of $D$
and $R$, respectively. Thus $K/E$ is an algebraic extension,
since $R$ is algebraic over $D$. Now, note that $K$ has a maximal
subring $V$ such that $\frac{1}{p}\notin V$ (for example
$V=D_{(p)}$). Hence $E$ has a maximal subring $W$ such that
$W\cap K=V$, by \cite[Proposition 2.1]{azkarm3}. Therefore
$\frac{1}{p}\notin W$. Thus we have $U(R)\nsubseteq W$ which
implies that $R$ is submaximal by the above proposition. Finally,
assume that $R$ is not algebraic over $D$, but by the preceding
comment we may suppose that $R$ is an integral domain too. Let
$X$ be a transcendental basis for $R$ over $D$. Thus $R$ is
algebraic over $D[X]$. Now note that $D[X]$ is a $UFD$ and $p$ is
an irreducible element in it. Hence we are done by the first part
of the proof. The final part is evident.
\end{proof}

The following fact also justifies the two cases proofs of
\cite[Proposition 2.10]{azkrm}.

\begin{rem}\label{remx1}
Let $R$ be a ring satisfying the conditions of the above theorem,
then there exists a maximal subring of $R$ which dose not contain
$\frac{1}{p}$. In particular, if $K$ is a field with zero
characteristic, then for any prime number $p$, there exists a
maximal subring $V_p$ of $K$ such that $\frac{1}{p}\notin V_p$.
Hence if $M$ is the unique nonzero prime ideal of $V_p$, we infer
that $Char(\frac{V_p}{M})=p$.
\end{rem}

The next three interesting facts are now immediate.

\begin{cor}\label{corx2}
Let $R$ be a $UFD$ and $S$ be a multiplicatively closed subset of
$R$ which contains a non-unit of $R$, then $R_S$ is submaximal.
\end{cor}

\begin{cor}\label{pp3}
Let $R$ be a ring with zero characteristic. If there exists a
natural number $n>1$ such that $\frac{1}{n}\in R$, then $R$ is
submaximal.
\end{cor}

\begin{cor}\label{corx6}
Let $D$ be an integral domain with zero characteristic and $X$ be
a set of independent indeterminates over it. Then for any $x\in
X$ and every natural number $n>1$, the ring
$\frac{D[X]}{(nx-1)D[X]}$ is submaximal.
\end{cor}

\begin{cor}\label{corx4}
If $R$ is a ring with $0=Char(R)\neq Char(\frac{R}{J(R)})$, then
any $R$-algebra $T$ is submaximal.
\end{cor}
\begin{proof}
Assume that $Char(\frac{R}{J(R)})=n$, thus $n\in J(R)$. Hence for
any $k\in \mathbb{Z}$, we have $1-kn\in U(R)\subseteq U(T)$ and
therefore we are done by Corollary \ref{pp3}.
\end{proof}

\begin{cor}\label{plemg1}
Let $R$ be a ring with zero characteristic which is not
submaximal. Then $\{Char(\frac{R}{M})\ |\ M\in Max(R)\
\}=\mathbb{P}$ and therefore $|Max(R)|$ is infinite.
\end{cor}
\begin{proof}
Since $R$ is not submaximal, we infer that $Char(\frac{R}{M})\neq
0$ for each maximal ideal $M$ of $R$, by Corollary \ref{pp3}.
Hence $\{Char(\frac{R}{M})\ |\ M\in Max(R)\ \}\subseteq
\mathbb{P}$. Now for each prime number $q$ we claim that there
exists a maximal ideal $M$ of $R$ with $Char(\frac{R}{M})=q$
which proves the lemma. To see this, we note that $qR\neq R$, by
Corollary \ref{pp3}. Consequently, there exists a maximal ideal
$M$ of $R$ with $qR\subseteq M$, i.e., $Char(\frac{R}{M})=q$.
\end{proof}

For more observations we need the following lemma.

\begin{lem}\label{qcufd}
Let $R$ be a ring and $x\in R$ is not algebraic over the prime
subring of $R$. Then at least one of the following conditions
holds.
\begin{enumerate}
\item If $Char(R)=0$, then there exists a prime ideal $Q$ of $R$
such that $R/Q$ contains a copy of $\mathbb{Z}[x]$.
\item If $Char(R)=n>0$, then for any prime divisor $p$ of $n$,
there exists a prime ideal $Q$ of $R$ such that $R/Q$ contains a
copy of $\mathbb{Z}_p[x]$.
\end{enumerate}
\end{lem}
\begin{proof}
If $R$ has zero (or prime) characteristic, we are done since
$Z[x]\setminus \{0\}$ is a multiplicatively closed set in $R$.
Now, suppose that $R$ has nonzero characteristic, say $n$, which
is also not a prime number. Assume that $p$ is a prime divisor of
$n$. Since $\dim \mathbb{Z}_n[x]=1$ and
$P=\frac{p\mathbb{Z}}{n\mathbb{Z}}[x]$ is a non-maximal prime
ideal of $\mathbb{Z}_n[x]$, hence we infer that $P$ is a minimal
prime ideal of $\mathbb{Z}_n[x]$. Thus, there exists a minimal
prime ideal $Q$ of $R$ such that $Q\cap\mathbb{Z}_n[x]=P$. Now we
have $\mathbb{Z}_p[x]\cong\frac{\mathbb{Z}_n[x]}{Q\cap
\mathbb{Z}_n[x]}\subseteq \frac{R}{Q}$ and therefore we are done.
\end{proof}

\begin{rem}
In fact in Corollary \ref{plemg1}, we see that if $R$ is not
submaximal and $\mathbb{Z}\subseteq R$, then $|Max(R)|\geq
|\mathbb{P}|$. We can generalize the previous fact to any
non-submaximal ring which contains a $UFD$ as follow. First, we
recall the reader that if $R$ is a ring with
$|Max(R)|>2^{\aleph_0}$, then $R$ is submaximal, see
\cite[Proposition 2.6 ]{azkrm}. Now assume that $D$ is a $UFD$
and let $Ir'(D)$ be a subset of $Ir(D)$ such that for any $p\neq
q$ in $Ir'(D)$, we have $pD+qD=D$. Now, if $R$ is a ring which
contains $D$, then either $R$ is submaximal or $|Ir'(D)|\leq
|Max(R)|\leq 2^{\aleph_0}$. To see this assume that $R$ is not
submaximal, then for any $q\in Ir'(D)$ we have $qR\neq R$, by
Theorem \ref{pp2} and hence there exists a maximal ideal $M_q$ of
$R$, such that $qR\subseteq M_q$. It is clear that whenever
$p\neq q$ in $Ir'(D)$, then we have $M_q\neq M_p$, and therefore
we are done. In particular, if $R$ is a non-submaximal ring with
nonzero characteristic, say $n$, which is not algebraic over
$\mathbb{Z}_n$, then $|Max(R)|$ is infinite. To see this note
that by part $(2)$ of the above lemma, for any prime divisor $p$
of $n$ there exists a prime ideal $Q$ of $R$ such that $R/Q$
contains a copy of $\mathbb{Z}_p[x]$. Hence we are done by the
first part of the proof.
\end{rem}

The following proof greatly simplifies the proof of \cite[Theorem
2.1 and Theorem 2.4]{azkarm4}.

\begin{cor}\label{p1}
\cite[Theorem 2.4]{azkarm4}. Let $R$ be a ring with a unit element
$x\in R$ which is not algebraic over $Z$. Then $R$ is submaximal
(in fact every $R$-algebra is submaximal).
\end{cor}
\begin{proof}
In view of Theorem \ref{pp2} and Lemma \ref{qcufd} we are done.
\end{proof}

\begin{cor}\label{p2}
Let $R$ be a ring. Then either $R$ is submaximal or $J(R)$ is
algebraic over $Z$.
\end{cor}

We need the following immediate corollary in the next section.

\begin{cor}\label{corx7}
Let $R$ be a ring which is not algebraic over $\mathbb{Z}$. Then
either $R$ is submaximal or for any non-algebraic element $x\in
R$ over $\mathbb{Z}$ and every natural number $n>1$, we have
$\mathbb{Z} \cap (nx-1)R \neq 0$.
\end{cor}
\begin{proof}
If $(nx-1)R=R$, then we are done by Corollary \ref{p1}, and if
not, then by using Corollary \ref{pp3}, we are done.
\end{proof}

\begin{cor}\label{corx8}
Let $D$ be a $PID$ and $R\supseteq D$ be an integral domain. If
$D$ is integrally closed in $R$ and $U(R)\neq U(D)$, then $R$ is
submaximal.
\end{cor}
\begin{proof}
Let $x\in U(R)\setminus U(D)$. If $x$ is not algebraic over $D$,
then we are done, by Corollary \ref{p1}. Hence assume that $x$ is
algebraic over $D$, thus there exists $b\in D$ such that $bx$ is
integral over $D$ and since $D$ is integrally closed in $R$, we
must have $bx=a\in D$. Therefore $x=\frac{a}{b}$. Now, since
$x\notin U(D)$, we infer that either $x\notin D$ or $x^{-1}\notin
D$. Therefore, in any case, there must exist $r,s\in D$ such that
$(r,s)=1$ and $z=\frac{r}{s}\in U(R)\setminus D$. Now since $D$ is
a $PID$, we infer that $\frac{1}{s}\in R$. Thus we are done, by
Theorem \ref{pp2}.
\end{proof}

\begin{lem}\label{propx3}
Let $R$ be a ring with nonzero characteristic $n$ which is square
free (in particular, if $R$ is reduced ring with nonzero
characteristic). Then either $R$ is submaximal or $U(R)$ is a
torsion group.
\end{lem}
\begin{proof}
Without lose of generality we may assume that $Char(R)=p$, where
$p$ is a prime number. Now suppose that $R$ is not submaximal,
then $U(R)$ must be algebraic over $\mathbb{Z}_p$, by Corollary
\ref{p1}. Assume that $x\in U(R)$, thus we infer that
$\mathbb{Z}_p[x]\cong\frac{\mathbb{Z}_p[t]}{I}$ where $I$ is a
nonzero ideal of the polynomial ring $\mathbb{Z}_p[t]$. Hence we
infer that $\mathbb{Z}_p[x]$ is a finite ring, and therefore $x$
is a torsion element. Thus $U(R)$ is a torsion group.
\end{proof}

We recall the reader that zero dimensional rings (in particular
von Neumann regular rings) with zero characteristic are
submaximal, see \cite[Corollary 3.11]{azkarm3}. We also have the
following.

\begin{prop}\label{regxx1}
Let $R$ be a von Neumann regular ring. Then either $R$ is
submaximal or $R$ is a periodic ring.
\end{prop}
\begin{proof}
If $R$ is not submaximal then by the above comment $R$ has
nonzero characteristic. Hence by the above lemma $U(R)$ is
torsion. But it is well-known that von Neumann regular rings are
unit regular, that is to say, for any $x\in R$, there exists $u\in
U(R)$ such that $x=x^2u$. Hence by the above lemma, if $u^n=1$,
then we have $x^n=x^{2n}$ and thus we are done.
\end{proof}

In fact the above result holds for any zero-dimensional ring $R$.
For proof note that if $R$ is not submaximal then $R$ has nonzero
characteristic, say $n$, and $R$ is integral over $\mathbb{Z}_n$,
by \cite[Corollary 3.14]{azkarm3}. Now note that for any $x\in
R$, the ring $\mathbb{Z}_n[x]$ is finite and hence we are done.
The next remark shows in some rings $R$, the group $U(R)$ may not
be torsion.

\begin{rem}\label{ntor}
Let $R$ be a ring.  If $R$ is von Neumann regular with zero
characteristic then clearly $U(R)$ is not torsion, by the proof of
the above proposition, since $R$ is not periodic. Also, if $R$ is
a ring with $J(R)\neq 0$, then $U(R)$ is not torsion. To see this
note that if $U(R)$ is torsion, then for any $0\neq x\in J(R)$,
there exists a natural number $n$ such that $(1+x)^n=1$. Hence we
infer that $x=0$, which is absurd.
\end{rem}

By Corollary \ref{p2}, if $R$ is a ring then either $R$ is
submaximal or every element of $J(R)$ is algebraic over $Z$. Now
we also have the following result.

\begin{prop}\label{jczd}
Let $R$ be a ring with zero characteristic and $J(R)\neq 0$. Then
either $R$ is submaximal or for any $x\in J(R)$ and $f(t)\in
\mathbb{Z}[t]$, if $f(x)=0$, then $f(0)=0$. In particular $J(R)$
consists of zero divisors.
\end{prop}
\begin{proof}
Assume that $R$ is not submaximal and $x\in J(R)$, $f(t)\in
\mathbb{Z}[t]$, and $f(x)=0$. Now since $x\in J(R)$, we infer that
if $u$ is one of the elements $1+f(0)$ or $1-f(0)$, then $u\in
U(R)\cap\mathbb{Z}$. Thus by Corollary \ref{pp3}, we have $u=1$
or $u=-1$. This implies that either $f(0)=0$, and therefore we are
done, or $f(0)\in\{2,-2\}$. But in the latter case, we have $2\in
J(R)$ and therefore $1-2n\in U(R)$, for each $n\in\mathbb{Z}$,
which is impossible by Corollary \ref{pp3}.
\end{proof}

The following is a generalization of \cite[Corollary
2.24]{azkarm4}.

\begin{cor}\label{p4}
Let $R$ be an integral domain with $J(R)\neq 0$. Then any
$R$-algebra $T$ is submaximal. In particular, any algebra over a
non-field $G$-domain is submaximal.
\end{cor}
\begin{proof}
If $R$ has nonzero characteristic or
$Char(R)=Char(\frac{R}{J(R)})=0$, then one can easily see that
$J(R)$ is not algebraic over $Z$ (note, if $0\neq x\in J(R)$ and
$a_nx^n+\cdots+a_1x+a_0=0$, where $n\in \mathbb{N}$, $a_i\in Z$
and $a_0\neq 0$, then we infer that $a_0\in J(R)$ which is
absurd). Therefore $U(R)$ is not algebraic over $Z$. Thus $U(T)$
is not algebraic over $Z$ and therefore $T$ is submaximal, by
Corollary \ref{p1}. Hence we may assume that $0=Char(R)\neq
Char(\frac{R}{J(R)})$ and hence $T$ is submaximal by Corollary
\ref{corx4}. The last part is now evident.
\end{proof}

\begin{rem}
One can prove the above corollary by using the proof of
Proposition \ref{jczd}, Lemma \ref{propx3} and Remark \ref{ntor}.
\end{rem}

\begin{prop}
Let $R\subseteq T$ be an extension of commutative rings with
lying-over property and $R$ is not Hilbert. Then $T$ is
submaximal.
\end{prop}
\begin{proof}
Let $P$ be a prime ideal in $R$ such that $P$ is not an
intersection of a family of maximal ideals in $R$. Now assume $Q$
is a prime ideal in $T$ lying over $P$. Thus $R/P\subseteq T/Q$
and since $J(R/P)\neq 0$, we infer that $T/Q$ is submaximal by
Corollary \ref{p4}.
\end{proof}

We recall the reader that if $R$ is a ring with
$|Max(R)|>2^{\aleph_0}$, then $R$ is submaximal, see
\cite[Proposition 2.6 ]{azkrm}.

\begin{cor}\label{p6}
Let $R$ be a ring. Then either $R$ is submaximal or it is a
Hilbert ring with $|Spec(R)|\leq 2^{2^{\aleph_0}}$.
\end{cor}
\begin{proof}
If $R$ is not submaximal, then for any prime ideal $P$ of $R$,
the integral domain $R/P$ is not submaximal too. Hence we infer
that $J(R/P)=0$, by Corollary \ref{p4}, i.e., $R$ is Hilbert and
therefore $P$ is an intersection of a set of maximal ideals of
$R$. Thus by the above comment we infer that $|Spec(R)|\leq
2^{2^{\aleph_0}}$.
\end{proof}

\begin{rem}
In fact if $R$ is not submaximal, then for any prime ideal $P$
and subring $S$ of $R$, the prime ideal $P\cap S$ is an
intersection of a family of maximal ideals of $S$. To see this
note that $R/P$ contains a copy of $S/(P\cap S)$, and since $R$
is not submaximal, we infer that $J(S/(P\cap S))=0$, by Corollary
\ref{p4}. Hence we are done.
\end{rem}

\begin{lem}\label{jaclem}
Let $R$ be a ring. Then at least one of the following conditions
holds,
\begin{enumerate}
\item There exists a maximal ideal $M$ of $R$, such that $R/M$ is
not an algebraic extension of a finite field (i.e., $R/M$ is not
absolutely algebraic field). In particular, $R/M$ and therefore
$R$ are submaximal.
\item For any subring $S$ of $R$, we have $J(S)\subseteq J(R)$.
\end{enumerate}
\end{lem}
\begin{proof}
If $(1)$ does not hold, then for any maximal ideal $M$ of $R$, the
field $R/M$ is algebraic over a finite field. Hence we infer that
every subring of $R/M$ is a field. Now note that if $S$ is a
subring of $R$, then $(S+M)/M$ is a subring of $R/M$ and therefore
$(S+M)/M$ is a field. Thus $S\cap M$ is a maximal ideal of $S$,
for any maximal ideal $M$ of $R$. This shows that $J(S)\subseteq
J(R)$. For the final part in $(1)$, note that by \cite[Theorem 1.8
]{azkrm}, if $R/M$ is not algebraic over a finite field, then
$R/M$ and therefore $R$ are submaximal.
\end{proof}

In \cite[Proposition 2.9]{azkrm} it is proved that if $R$ is a
ring with $|R/J(R)|>2^{2^{\aleph_0}}$, then $R$ is submaximal.

\begin{cor}\label{p7}
Let $R$ be a reduced ring. If either $J(R)\neq 0$ or
$|R|>2^{2^{\aleph_0}}$, then $R$ is submaximal. Moreover, every
$R$-algebra $T$, is submaximal too.
\end{cor}
\begin{proof}
If $R$ is not submaximal, then by Corollary \ref{p6}, $R$ is
Hilbert ring and therefore $J(R)=N(R)$. Hence we infer that
$J(R)=0$ and by the above comment also we have $|R|\leq
2^{2^{\aleph_0}}$ which are absurd. For the final part note that
$T/N(T)$ contains a copy of $R$, hence by our assumptions, either
by the above lemma $J(T/N(T))\neq 0$, or
$|T/N(T)|>2^{2^{\aleph_0}}$. Thus by the first part, $T/N(T)$ and
therefore $T$ are submaximal.
\end{proof}

Hence by the above corollary if $T$ is a non-submaximal ring,
then for any reduced subring $R$ of $T$ we have $J(R)=0$. More
generally, for any subring $R$ of $T$ we have $N(R)=J(R)$. To see
this, note that $R+N(T)$ is a subring of $T$. Now since $T/N(T)$
contains a copy of $R/N(R)$, we infer that $J(R/N(R))=0$,  hence
we are done. The following is also interesting.

\begin{cor}
Let $R$ be a reduced ring with zero characteristic, then either
$R$ is submaximal or $\bigcap_{p\in \mathbb{P}} Rp=0$.
\end{cor}
\begin{proof}
If $R$ is not submaximal then by Corollary \ref{plemg1}, we infer
that $\bigcap_{p\in \mathbb{P}} Rp\subseteq J(R)$. But by the
above corollary we also have $J(R)=0$. Hence we are done.
\end{proof}

For a ring $R$, let $CoHt_1(R)$ be the set of all prime ideal $P$
of $R$, such that if $P\subsetneq Q$, where $Q\in Spec(R)$, then
$Q\in Max(R)$.

\begin{prop}
Let $R$ be a noetherian Hilbert ring, then any $P\in CoHt_1(R)$
is an intersection of infinite countably many maximal ideals.
Consequently, if $R$ is one dimensional noetherian (Hilbert)
domain, then for any infinite family $\A$ of maximal ideals of
$R$, we have $\bigcap \A=0$. In particular,
\begin{enumerate}
\item For any noetherian ring $R$, either $R$ is submaximal
or $|CoHt_1(R)|\leq 2^{\aleph_0}$.
\item If $R$ is noetherian domain and $dim(R)=2$, then either
$R$ is submaximal or $|Spec(R)|\leq 2^{\aleph_0}$.
\end{enumerate}
\end{prop}
\begin{proof}
Let $\{M_i\}_{i\in I}$ be a family of maximal ideals of $R$ such
that $P=\bigcap_{i\in I}M_i$. We first show that for any infinite
countable subfamily $\{M_n\}_{n\in\mathbb{N}}$ of $\{M_i\}_{i\in
I}$ we have $P=\bigcap_{n\in \mathbb{N}} M_n$. Assume
$A=\bigcap_{n\in \mathbb{N}} M_n$, hence we infer that
$P\subseteq A\subseteq M_n$, for all $n\in \mathbb{N}$. Hence, if
$P\neq A$, then $Min(R/A)$ is infinite which is a contradiction,
thus $P=A$ and we are done. Now, if furthermore $R$ is one
dimensional domain, then any infinite intersection of maximal
ideals is zero, by the previous proof. To see $(1)$, assume that
$R$ is not submaximal, hence $|Max(R)|\leq 2^{\aleph_0}$, by the
comment preceding Corollary \ref{p6}, and since any $P\in
CoHt_1(R)$ is a countable intersection of maximal ideals, we
infer that $|CoHt_1(R)|\leq 2^{\aleph_0}$. Part $(2)$ is now
evident.
\end{proof}

We recall that each zero dimensional ring with nonzero
characteristic which is not integral over its prime subring, is
submaximal, see \cite[Corollary 3.14]{azkarm3}. The following is
a generalization of the existence of maximal subrings in artinian
rings, see \cite{azkrm2}.

\begin{cor}\label{p8}
Let $R$ be a semi-local ring. Then either $R$ is submaximal or
$R$ has nonzero characteristic, say $n$, which is integral over
$\mathbb{Z}_n$ (thus $R$ is zero-dimensional). In particular,
every semi-local ring with zero characteristic is submaximal.
Consequently,
\begin{enumerate}
\item Non-submaximal semi-local integral domains are exactly
non-submaximal fields.

\item Every non-submaximal noetherian semi-local ring, is
countable artinian.
\end{enumerate}
\end{cor}
\begin{proof}
If $Char(R)=0$, then we are done by Corollary \ref{plemg1}. Hence
assume that $R$ has nonzero characteristic. If $R$ is not
submaximal, then $R$ is a Hilbert ring by Corollary \ref{p6}.
Therefore every non-maximal prime ideal of $R$ is an intersection
of infinitely many maximal ideals. Hence we infer that $R$ is
zero dimensional, since $|Max(R)|<\aleph_0$. Now by the above
comment we infer that $R$ must be integral over its prime
subring. For part $(1)$, we note that the prime subring of an
integral domain with nonzero characteristic is a field; and for
$(2)$ note that $R$ is a zero-dimensional ring. Hence $R$ is
artinian. Thus by \cite[Proposition 2.4]{azkrm2}, $R$ must be
countable too.
\end{proof}

\begin{prop}\label{p13}
Let $R_1\subseteq R_2$ be extension of rings. Assume that $R_1$
is semi-local. Then either $R_2$ is submaximal or $R_1$ is
zero-dimensional. In other words, every algebra over a semi-local
ring which is not zero dimensional, is submaximal.
\end{prop}
\begin{proof}
First note that, if $P$ is a prime ideal of $R_2$, then the ring
$R_2/P$ contains a copy of $S=R_1/(R_1\cap P)$. Hence if $J(S)\neq
0$, then $R_2/P$ and therefore $R_2$ are submaximal by Corollary
\ref{p4}. If not, then we infer that $P\cap R_1$ is a maximal
ideal of $R_1$, since $R_1$ is semi-local. Hence, we may assume
that for any prime ideal $P$ of $R_2$, $R_1\cap P$ is a maximal
ideal in $R_1$. Now, if $Q$ is a prime ideal in $R_1$, then there
exists a prime ideal $P$ of $R_2$ such that $P\cap R_1\subseteq
Q$ (note, there exists a prime ideal $P$ of $R_2$ with
$P\cap(R_1\setminus Q)=\emptyset$). Hence we infer that $Q=P\cap
R_1$ and therefore $Q$ is maximal in $R_1$. Hence $R_1$ is a zero
dimensional ring.
\end{proof}

We recall the reader that a ring $R$ is called residue finite if
$R/I$ is a finite ring for every nonzero ideal $I$ of $R$. It is
clear that if $R$ is a residue finite ring, then $dim(R)\leq 1$
and in fact $dim(R)=1$ if and only if $R$ is a non-field integral
domain. In the next theorem we give the structure of
non-submaximal residue finite rings.

\begin{thm}
Let $R$ be a residue finite ring which is not submaximal. Then
$R$ is a countable principal ideal ring. Moreover, exactly one of
the following holds:
\begin{enumerate}
\item If $dim(R)=1$, then $R=U(R)\mathbb{Z}$  and $R$
is algebraic over $\mathbb{Z}$.

\item If $dim(R)=0$, then $R$ is an artinian ring with nonzero
characteristic, say $n$, which is also integral over
$\mathbb{Z}_n$. Moreover, $R$ has only finitely many ideals. In
particular, if $R$ is reduced then $R$ is finite.
\end{enumerate}
\end{thm}
\begin{proof}
First note that since $R$ is not submaximal then for any nonzero
ideal $I$ of $R$ we infer that $R/I\cong\mathbb{Z}_m$ for some
natural number $m$ (note, it is clear that all finite rings
except $\mathbb{Z}_n$, up to isomorphism, where $n$ is a natural
number, are submaximal). This shows that $I=Rm$ and therefore $R$
is a principal ideal ring. Now by the above comment we have two
cases, either $dim(R)=1$ or $dim(R)=0$. First assume that
$dim(R)=1$ and therefore $R$ is a non-field integral domain.
Hence we have two cases.
\begin{enumerate}
\item If $R$ has nonzero characteristic, say $p$ (where
$p\in\mathbb{P}$), then we infer that for any nonzero ideal $I$
of $R$ we have $R/I\cong \mathbb{Z}_p$, which is absurd, by the
preceding comment (note, in this case $R\cong \mathbb{Z}_p$ which
is impossible).
\item If $R$ has zero characteristic. Then $R$ is a $PID$ with
$Ir(R)=\mathbb{P}$, by the first part of the proof. Hence we
infer that $R=U(R)\mathbb{Z}$. Also note that by Corollary
\ref{p1}, $U(R)$ is algebraic over $\mathbb{Z}$. Therefore $U(R)$
is countable and hence $R$ is countable too. Thus we are done.
\end{enumerate}

Now assume that $R$ is zero-dimensional ring. Thus $R$ is
artinian, since $R$ is noetherian (note, every ideal of $R$ is
principal). Thus by \cite[Proposition 2.4]{azkrm2}, $R$ is
countable and has a nonzero characteristic, say $n$, which is also
integral over $\mathbb{Z}_n$, by \cite[Corollary 2.5]{azkrm2}.
Moreover, by the first part of the proof since every nonzero
ideal of $R$ has the form $I=Rm$ where $m|n$, we infer that $R$
has only finitely many ideals. Also note that if $R$ is reduced,
then by Corollary \ref{p7}, we infer that $J(R)=0$ and therefore
$R\cong \mathbb{Z}_n$ (where $n$ is square free) and hence we are
done.
\end{proof}

\begin{prop}
Let $D$ be an integral domain and $S$ be a multiplicatively
closed set in it such that $S\nsubseteq U(D)$. If $D_S$ is not
submaximal then the following conditions hold.
\begin{enumerate}
\item $D$ has zero characteristic. $S$ is algebraic over
$\mathbb{Z}$ and therefore $|S|\leq \aleph_0$. In particular
$\mathbb{Z}$ is not integrally closed in $D_S$.
\item There exists an infinite subset $\mathcal{M}$ of $Max(D)$ such that $Max(D_S)=\{Q_S\ |\ Q\in
\mathcal{M}\}$. In particular, $\bigcap \mathcal{M}=0$.
\item For any non-maximal prime ideal $P_S$ of $D_S$, either
$Char(\frac{D}{P})=0$ or $\frac{D_S}{P_S}\cong\frac{D}{P}$.
\end{enumerate}
\end{prop}
\begin{proof}
Since $D_S$ is not submaximal then by Corollary \ref{p1}, we
infer that $S$ is algebraic over $Z$. Hence if $Char(D)\neq 0$,
then $S\subseteq U(D)$ which is absurd. Thus $D$ has zero
characteristic. Hence by Corollary \ref{corx8}, $\mathbb{Z}$ is
not integrally closed in $D_S$. Now assume $P_S$ is a maximal
ideal in $D_S$ and $P\in Spec(D)\setminus Max(D)$. Thus we have
$\frac{D_S}{P_S}\cong (\frac{D}{P})_{\bar{S}}=Frac(\frac{D}{P})$,
where $\bar{S}=\{s+P\ |\ s\in S\}$ and $Frac(\frac{D}{P})$ is the
quotient field of $\frac{D}{P}$, and since $P$ is not maximal we
infer that $Frac(\frac{D}{P})$ is submaximal by Corollary
\ref{p1}, and therefore $D_S$ is submaximal which is absurd.
Hence $P\in Max(D)$. Also, note that by Corollary \ref{p8},
$Max(D_S)$ is infinite since $D_S$ is not submaximal; and by
Corollary \ref{p4} we have $J(D_S)=0$ and therefore $\bigcap
\mathcal{M}=0$. Finally, for part $(3)$, assume that
$Char(\frac{D}{P})=q>0$, then by part $(1)$, either
$\bar{S}\subseteq U(\frac{D}{P})$ and therefore
$\frac{D_S}{P_S}\cong (\frac{D}{P})_{\bar{S}}=\frac{D}{P}$ and we
are done; or $\bar{S}\nsubseteq U(\frac{D}{P})$ and therefore
$(\frac{D}{P})_{\bar{S}}$ is submaximal. Thus $D_S$ is submaximal
which is absurd.
\end{proof}

Note that in the above proposition clearly for any maximal ideal
$Q_S$ of $D_S$ we also have $\frac{D_S}{Q_S}\cong\frac{D}{Q}$.
More generally, if $R$ is a ring and $S$ be a multiplicatively
closed set in $R$, then the non-submaximality of $R_S$ implies
that every maximal ideal of $R_S$ has the form $P_S$ for some
maximal ideal $P$ of $R$, by the preceding proof. In particular,
if $R$ has nonzero characteristic then one can easily see that,
by a similar proof, for every prime ideal $P_S$ of $R_S$ we have
$\frac{R_S}{P_S}\cong\frac{R}{P}$. The following is a
generalization of \cite[Theorem 3.2 ]{azkarm4}.

\begin{thm}
Let $R$ be a ring and $S$ be a multiplicatively closed set in
$R$, such that $R_S$ is semi-local. Then at least one of the
following holds.
\begin{enumerate}
\item $R_S$ is submaximal.
\item $Spec(R_S)=Max(R_S)=\{P_S\ |\ P\in \mathcal{M}\}$, where
$\mathcal{M}$ is a finite subset of $Min(R)\cap Max(R)$.
\end{enumerate}
In particular if $(2)$ holds and $R_S$ is submaximal, then $R$ is
submaximal too.
\end{thm}
\begin{proof}
Assume that $R_S$ is not submaximal, then by the above comment
$Max(R_S)=\{P_S\ |\ P\in \mathcal{M}\}$, where $\mathcal{M}$ is a
finite subset of $Max(R)$. But since $R_S$ is a semi-local
non-submaximal ring, then we infer that $R_S$ is
zero-dimensional, by Corollary \ref{p8}. Hence
$\mathcal{M}\subseteq Min(R)$. Now assume that $(2)$ holds and
$R_S$ is submaximal. Thus by \cite[Theorem 2.26]{azkarm4}, at
least one of the following holds (note $R_S$ is zero-dimensional).
\begin{enumerate}
\item There exists a maximal ideal $P_S$ of $R_S$, such that
$R_S/P_S$ is submaximal. But, since $R_S/P_S\cong R/P$, we infer
that $R/P$ and therefore $R$ are submaximal.
\item There exist distinct maximal ideals $P_S$ and $Q_S$ of $R_S$
such that $R_S/P_S\cong R_S/Q_S$. Hence similar to $(1)$, we infer
that $R/P\cong R/Q$ and therefore $R$ is submaximal, by
\cite[Theorem 2.2]{azarang}.
\item There exist an ideal $I_S$ and a maximal ideal $P_S$ of
$R_S$, such that $(P_S)^2\subseteq I_S\subseteq P_S$ and
$R_S/I_S\cong K[x]/(x^2)$, for some field $K$. Hence we infer
that $I$ is a $P$-primary ideal in $R$. Therefore $R/I$ is a local
ring with unique prime ideal $P/I$. Thus $R_S/I_S\cong
(R/I)_{\bar{S}}\cong (R/I)_{P/I}=R/I$, where $\bar{S}=\{s+I\ |\  s
\in S\}$, see \cite[P. 24, Ex. 7]{kap}. Hence $R/I$ and therefore
$R$ are submaximal.
\end{enumerate}
\end{proof}

The following remark which is a generalization of \cite[Corollary
1.15]{azkrm} is interesting.

\begin{rem}\label{p15}
Let $\mathcal{F}$ be the set of all fields, up to isomorphism,
which are not submaximal and let $\mathcal{D}$ be the class of all
integral domains (or reduced rings), up to isomorphism, which are
not submaximal. Now for any $D\in \mathcal{D}$ we have the
following facts:
\begin{enumerate}
\item For any $M\in Max(D)$, we have $D/M\in \mathcal{F}$.
\item For any $M, N\in Max(D)$, with $M\neq N$ we have $D/M\ncong
D/N$ (For otherwise, by \cite[Theorem 2.2]{azarang}, $D$ is
submaximal). In other words there exists an injection $\Phi_D$
from $Max(D)$ into $\mathcal{F}$, sending $M$ into $D/M$.
\item $|Max(D)|\leq |\mathcal{F}|$, by \cite[Proposition 2.6]{azkrm} or $(2)$.
\item $J(D)=0$, by Corollary \ref{p4} or \ref{p7}.
\item Hence we have the natural rings embedding $D\hookrightarrow\prod_{M\in Max(D)} D/M\hookrightarrow
\prod_{E\in\mathcal{F}}E$ (i.e., every non-submaximal integral
domain (or reduced ring) can be embedded in
$\prod_{E\in\mathcal{F}}E$).
\end{enumerate}
Now, for any $D\in\mathcal{D}$, let $RdMax(D)=Im(\Phi_D)$. Two non
submaximal integral domains $D_1$ and $D_2$ are called
$RdMax$-equivalent, if $RdMax(D_1)=RdMax(D_2)$. Now, let
$\mathcal{D}'$ be the set of equivalent classes of this relation.
We claim that $|\mathcal{D}'|\leq
2^{|\mathcal{F}|}=2^{2^{\aleph_0}}$ and $\mathcal{F}\subseteq
\mathcal{D}'$. To show this, it is clear that
$\mathcal{F}\subseteq \mathcal{D}'$. Also note that for any
$[D]\in \mathcal{D}'$, the function that send $[D]$ into
$RdMax(D)$ is well-defined and one-one from $\mathcal{D}'$ into
$P(\mathcal{F})$, the set of all subsets of $\mathcal{F}$. Hence
we are done, since $|\mathcal{F}|=2^{\aleph_0}$, by
\cite[Corollary 1.15]{azkrm}.
\end{rem}

\section{Submaximal Integral Domains}

In \cite[Corollary 1.3]{azkrm2}, it is proved that uncountable
fields, are submaximal. The following interesting result is a
generalization of this fact.

\begin{thm}\label{t1}
Let $R$ be an uncountable $UFD$, then $R$ is submaximal.
\end{thm}
\begin{proof}
If $U(R)$ is uncountable, then we are done by Corollary \ref{p1}.
Hence we may assume that $U(R)$ is countable. Thus $|Ir(R)|=|R|$.
Now, note that there exists a $p\in Ir(R)$ such that $1-p\notin
U(R)$. Hence let $p$ be an element in $Ir(R)$, such that
$1-p\notin U(R)$ and $q_0\in Ir(R)$ such that $q_0|1-p$. Thus
$q_0\in A=\{q\in Ir(R)\ |\ pR+qR=R\ \}$. Now we show that $A$
must be an uncountable set. Let us assume that $A$ is countable
and put $B=\{pq+1\ |\ q\in Ir(R)\setminus A \}$. It is clear that
$B$ is an uncountable set and therefore there exists a non-unit
element $x\in B$ such that $x$ has an irreducible divisor $q'\in
Ir(R)\setminus A$ (note, $U(R)$ and $A$ are countable, thus the
set of all elements which are of the form $uq_1\cdots q_n$, where
$u\in U(R)$ and $q_i\in A$, $n\in \mathbb{N}\cup\{0\}$ must be a
countable set). Hence $q'R+pR=R$ and $q'\notin A$, which is a
contradiction. Thus $A$ must be uncountable. Now for any $q\in
A$, $p+(q)$ is a unit in the ring $R/(q)$, hence if there exists
$q\in A$ such that $p+(q)$ is not algebraic over the prime
subring of $R/(q)$, then by Corollary \ref{p1}, $R/(q)$ and
therefore $R$ are submaximal. Consequently, we may assume that
for any $q\in A$, $p+(q)$ is algebraic over the prime subring of
$R/(q)$. Thus for any $q\in A$, $Z[p]\cap (q)\neq 0$, where $Z$
is the prime subring of $R$. But $Z[p]$ is a countable set and
the set $\{(q)\}_{q\in A}$ is uncountable, thus there exists a
nonzero element $f\in Z[p]$ which belongs to an infinite (in fact
uncountable) number of $(q)$, where $q\in A$, which is a
contradiction. This proves the theorem.
\end{proof}

\begin{cor}\label{t2}
Let $R$ be a non-submaximal non-field $PID$, then $R$ is countable
and $|Ir(R)|=|R|$.
\end{cor}
\begin{proof}
By the above theorem, $R$ and $Ir(R)$ are countable. Now note that
if $Ir(R)$ is finite then $R$ is a $G$-domain and therefore $R$ is
submaximal by Corollary \ref{p4}, hence we are done.
\end{proof}

\begin{cor}
Every localization of an uncountable $UFD$ is submaximal.
\end{cor}

\begin{prop}\label{yy1}
Let $D$ be an uncountable atomic (or noetherian) domain. Assume
that there exists an irreducible element $p$ of $D$ such that
$1-p\notin U(D)$ and every element of $Z[p]$ has at most
countably many (irreducible) divisors. Then $D$ is submaximal. In
particular, if $D$ is an uncountable atomic (or noetherian)
domain such that every element of it has at most countably many
(irreducible) divisors, then $D$ is submaximal. Consequently,
every uncountable noetherian $idf$-domain is submaximal.
\end{prop}
\begin{proof}
Note that any noetherian integral domain is an atomic domain, and
by using the proof of the previous theorem word-for-word, one can
easily complete the proof.
\end{proof}

\begin{thm}\label{yy2}
Let $R$ be an uncountable atomic (or noetherian) integral domain
with zero characteristic. If every $n\in\mathbb{N}$ has at most
countably many (irreducible) divisors, then $R$ is submaximal.
\end{thm}
\begin{proof}
We may assume that $U(R)$ is algebraic over $\mathbb{Z}$ and
therefore it is countable, by Corollary \ref{p1}. Hence we infer
that $|Ir(R)|=|R|$ and therefore $Ir(R)$ is uncountable. Let $X$
be a transcendental basis for $R$ over $\mathbb{Z}$. Now, if
there exist a natural number $n>1$ and $x\in X$ such that
$\mathbb{Z}\cap (nx-1)R=0$, then $R$ is submaximal by Corollary
\ref{corx7}. Hence we may assume that $\mathbb{Z}\cap (nx-1)R \neq
0$ for any natural number $n>1$ and $x\in X$. Since $X$ is
uncountable and the number of ideals of $\mathbb{Z}$ is
countable, we infer that there exists an uncountable subset $Y$
of $X$ such that for any $y\in Y$ we have $\mathbb{Z}\cap
(ny-1)R=m\mathbb{Z}$ for some fixed natural numbers $n>1$ and
$m$. Hence for any $y\in Y$ we have $ny-1|m$. Now we show that
$m$ has uncountably many irreducible divisors. Assume that
$P=\{q\in Ir(R)\ : q|ny-1, \text{for some}\ y\in Y\}$. If $P$ is
countable, then we infer that $\{ny-1\ :\ y\in Y\}$ is countable
too (note $U(R)$ is countable) which is a contradiction. Hence
$P$ is uncountable. Now note that any $q\in P$ is an irreducible
divisor of $m$, i.e., $m$ has uncountable many irreducible
divisors, which is a contradiction. Hence, for any natural number
$n>1$, the set $\{ x\in X\ |\ \mathbb{Z}\cap(nx-1)R \neq 0\}$ is
countable, and therefore $R$ is submaximal, by Corollary
\ref{corx7}.
\end{proof}

For more observations we need the following definition, see
\cite{ayache}.

\begin{defn}
An extension $R\subseteq T$ of rings is called residually
algebraic extension, if for any prime ideal $Q$ of $T$, the ring
$T/Q$ is algebraic over $R/(Q\cap R)$.
\end{defn}

One can easily see that if $R\subseteq T$ is a residually
algebraic extension then $T$ must be algebraic over $R$. Also see
\cite{ayache} for more interesting results about residually
algebraic extensions. In particular, see \cite[Section 4,
b-Maximal subrings]{ayache} which contains interesting results
related to the subject of this paper. The following lemma is
needed for the next theorem.

\begin{lem}
Let $R\subseteq T$ be a residually algebraic extension of rings
and $R$ has finite dimension. Then $T$ has finite dimension too
and we have $dim(T)\leq dim(R)$.
\end{lem}
\begin{proof}
Assume that $n=dim(R)$. First suppose that $T$ is an integral
domain, and we prove the lemma by induction on $n$. If $n=0$, then
$R$ is a field and therefore $T$ is a field too, hence we are
done. Thus assume that $n\geq 1$ and the lemma holds for any
residually algebraic extension (of integral domains) $R\subseteq
T$ with $dim(R)<n$. Now assume that $R\subseteq T$ is a residually
algebraic extension of integral domains with $dim(R)=n$. Hence
for any nonzero prime ideal $Q$ of $T$, the extension $R/(Q\cap
R)\subseteq T/Q$ is also a residually algebraic extension of
integral domains and $dim(R/(Q\cap R))<n$ (note that $T$ is
algebraic over $R$ and $Q\neq 0$, hence $Q\cap R\neq 0$). Hence
we infer that $dim(T/Q)<n$. This immediately implies that
$dim(T)\leq n$ and therefore we are done. Now assume that
$R\subseteq T$ be any residually algebraic extension, $dim(R)=n$
and $Q$ be a prime ideal of $T$. Hence $R/(Q\cap R)\subseteq T/Q$
is a residually algebraic extension of integral domains and
$dim(R/(Q\cap R))\leq n$. Thus by the first part of the proof we
infer that $dim(T/Q)\leq n$ and since the latter inequality holds
for any prime ideal $Q$ of $T$, we must have $dim(T)\leq n$.
Therefore we are done.
\end{proof}

The following is now in order.

\begin{thm}
Let $R$ be a noetherian integral domain with $tr.deg_Z(R)=n<
\aleph_0$ and assume that $X$ is a transcendental basis for $R$
over $Z$. Moreover let $Z[X]\subseteq R$ be a residually
algebraic extension and at least one of the following holds.
\begin{enumerate}
\item If $Char(R)=0$ and $n\geq 1$, then for any maximal ideal $M$
of $R$ we have $ht(M)\geq n+1$.

\item If $Char(R)=p>0$ and $n\geq 2$, then for any maximal ideal
$M$ of $R$ we have $ht(M)\geq n$.

\end{enumerate}
Then $R$ is submaximal.
\end{thm}
\begin{proof}
$(1)$ Let $x\in X$, if $R(2x-1)=R$, then we are done, by Corollary
\ref{p1}. Hence assume that $R(2x-1)\neq R$ and $P$ be a prime
ideal of $R$ which is minimal over $R(2x-1)$. Thus by the Krull's
principal ideal theorem we infer that $ht(P)=1$ and therefore by
our assumption $P$ is not a maximal ideal in $R$. Hence
$(0)\subsetneq (2x-1)\mathbb{Z}[X]\subseteq P\cap\mathbb{Z}[X]$.
Thus we have two cases. First, if $(2x-1)\mathbb{Z}[X]=
P\cap\mathbb{Z}[X]$, then
$\frac{\mathbb{Z}[X]}{(2x-1)\mathbb{Z}[X]}\subseteq \frac{R}{P}$
and therefore $\frac{1}{2}\in U(\frac{R}{P})$, hence we are done,
by Corollary \ref{pp3}. Thus we may assume that
$Q=P\cap\mathbb{Z}[X]\neq (2x-1)\mathbb{Z}[X]$. Therefore
$ht(Q)\geq 2$ and since $dim(\mathbb{Z}[X])=n+1$, we infer that
$dim(\frac{\mathbb{Z}[X]}{Q})\leq n-1$. But
$\frac{\mathbb{Z}[X]}{Q}\subseteq \frac{R}{P}$ is a residually
algebraic extension, hence by the above lemma conclude that
$dim(\frac{R}{P})\leq n-1$. Now since $ht(P)=1$, the latter
inequality immediately implies that $ht(M)\leq n$, for any
maximal ideal $M\supseteq P$, which is absurd. Thus we are done.\\

$(2)$ Let $x,y\in X$ and $x\neq y$. If $R(1-xy)=R$, then we are
done by Corollary \ref{p1}. Hence assume that $R(1-xy)\neq R$ and
$P$ be a prime ideal of $R$ which is minimal over $R(xy-1)$. Thus
by the Krull's principal ideal theorem we infer that $ht(P)=1$ and
therefore by our assumption $P$ is not a maximal ideal in $R$.
Hence $(0)\subsetneq (xy-1)\mathbb{Z}_p[X]\subseteq
P\cap\mathbb{Z}_p[X]$. Thus we have two cases. First, if
$(xy-1)\mathbb{Z}_p[X]= P\cap\mathbb{Z}_p[X]$, then
$\frac{\mathbb{Z}_p[X]}{(xy-1)\mathbb{Z}_p[X]}\subseteq
\frac{R}{P}$ and therefore $x+(xy-1)\mathbb{Z}_p[X]\in
U(\frac{R}{P})$, hence we are done by Corollary \ref{p1} (note,
$x+(xy-1)\mathbb{Z}_p[X]$ is not algebraic over $\mathbb{Z}_p$,
since $\mathbb{Z}_p[X]$ is a $UFD$). Thus assume that
$Q=P\cap\mathbb{Z}_p[X]\neq (xy-1)\mathbb{Z}_p[X]$. Therefore
$ht(Q)\geq 2$ and since $dim(\mathbb{Z}_p[X])=n$ we infer that
$dim(\frac{\mathbb{Z}_p[X]}{Q})\leq n-2$. But
$\frac{\mathbb{Z}_p[X]}{Q}\subseteq \frac{R}{P}$ is a residually
algebraic extension, hence by the above lemma we infer that
$dim(\frac{R}{P})\leq n-2$. Now since $ht(P)=1$, the latter
inequality immediately implies that $ht(M)\leq n-1$, for any
maximal ideal $M\supseteq P$, which is absurd. Thus we are done.
\end{proof}

\begin{prop}\label{t7}
Let a non-singleton $X\neq \emptyset$ be a set of independence
indeterminates in a noetherian ring $R$ over $Z$, where $Z$ is
the prime subring of $R$. If $Char(R)\in \mathbb{P}\cup\{0\}$ and
$R$ is integral over $Z[X]$, then $R$ is submaximal.
\end{prop}
\begin{proof}
Let $x, y\in X$ and $x\neq y$. If $R(1-xy)=R$, then we are done.
Hence assume that $P$ is a minimal prime ideal of $R(1-xy)$. Hence
$ht(P)\leq 1$, by the Krull's principal ideal theorem. Thus
$ht(P\cap Z[X])\leq 1$. But $(1-xy)Z[X]$ is a prime ideal in
$Z[X]$, which is contained in $P\cap Z[X]$. So we infer that
$(1-xy)Z[X]=P\cap Z[X]$ and therefore $T=Z[X]/(1-xy)Z[X]\subseteq
R/P$. Now $\bar{x}$ and $\bar{y}$ are units in $T$, which are not
algebraic over the prime subring of $T$ (note, $Z[X]$ is a
$UFD$). Hence $R/P$ has unit elements which are not algebraic
over its prime subring and therefore we are done, by Corollary
\ref{p1}.
\end{proof}

We conclude this article with the following fact about Dedekind
domains.

\begin{prop}\label{t8}
Let $D$ be an uncountable Dedekind domain with countable set of
maximal ideals. Then $U(D)$ is uncountable. In particular, $D$ is
submaximal.
\end{prop}
\begin{proof}
Let $U(D)$ be countable and seek a contradiction. It is now clear
that $Ir(D)$ is uncountable. Hence we infer that the set of
principal ideals of $D$ is uncountable. But since $D$ is a
Dedekind domain, every nonzero ideal of $D$ is a finite product of
prime ideals. Since the set of prime ideals of $D$ is countable,
we infer that the set of ideals $D$ is countable too, which is a
contradiction. Thus $U(D)$ is uncountable and therefore we are
done, by Corollary \ref{p1}.
\end{proof}


\begin{thebibliography}{mm}

\bibitem{anderson}
D.D. Anderson and Bernadette Mullins, Finite factorization
domains, Proc. A.M.S. {\bf{124}} (2) (2008) 389-396.

\bibitem{ayache}
A. Ayache and A. Jaballah, Residually algebraic pairs of rings,
Math Z. {\bf 225} (1997) 49-65.

\bibitem{azarang}
A. Azarang, On maximal subrings, Far East J. Math. Sci (FJMS)
{\bf{32}} (1) (2009) 107-118.


\bibitem{azkrm}
A. Azarang, O.A.S. Karamzadeh, Which fields have no maximal
subrings?, Rend. Sem. Mat. Univ. Padova, {\bf 126} (2011) 213-228.


\bibitem{azkrm2}
A. Azarang, O.A.S. Karamzadeh, On the existence of maximal
subrings in commutative artinian rings, J. Algebra Appl. {\bf 9}
(5) (2010) 771-778.

\bibitem{azkarm3}
A. Azarang, O.A.S. Karamzadeh, On Maximal Subrings of Commutative
Rings, to appear in Algebra Colloquium.


\bibitem{azkarm4}
A. Azarang, O.A.S. Karamzadeh, Most Commutative Rings Have
Maximal Subrings, to appear in Algebra Colloquium.




\bibitem{bl1}
H.E. Bell and F. Guerriero, Some condition for finiteness and
commutativity of rings, J. Math. Math. Sci. {\bf 13} (3) (1990)
535-544.

\bibitem{bl2}
H.E. Bell and A.A. Klein, On finiteness of rings with finite
maximal subrings, J. Math. Math. Sci. {\bf 16} (2) (1993) 351-354.






\bibitem{cahen1}
P.J. Cahen, D.E. Dobbs and T.J. Lucas, Valuative domains, J.
Algebra Appl. {\bf 9} (1) (2010) 43-72.


\bibitem{cahen2}
P.J. Cahen, D.E. Dobbs and T.J. Lucas, Characterizing minimal
ring extensions, Rocky Mountain Journal of Math {\bf 41} (4)
(2011) 1081-1125.


\bibitem{db4}
David E. Dobbs, Every commutative ring has a minimal ring
extension, Comm. Algebra {\bf 34} (2006) 3875-3881.

\bibitem{db8}
D.E. Dobbs and J. Shapiro, A classification of the minimal ring
extensions of an integral domain, J. Algebra, 305 (2006), 185-193.

\bibitem{db9}
D.E. Dobbs and J. Shapiro, A classification of the minimal ring
extensions of certain commutative rings, J. Algebra, 308 (2007),
800-821.




\bibitem{frd}
D. Ferrand, J.-P. Olivier, Homomorphismes minimaux d'anneaux, J.
Algebra {\bf 16} (1970) 461-471.





\bibitem{kap}
I. Kaplansky, Commutative Rings, Revised edn. (University of
Chicago Press, Chicago 1974).


\bibitem{klein}
Abraham A. Klein, The finiteness of a ring with a finite maximal
subrings, Comm. Algebra {\bf 21} (4) (1993) 1389-1392.


\bibitem{kor}
S.S. Korobkov, Finite Rings with Exactly Two Maximal Subrings,
Russian Mathematics (Iz. VUZ), {\bf 55} (6) (2011) 46–52.


\bibitem{laffey}
Thomas J. Laffey, A finiteness theorem for rings, Proc. R. Ir.
Acad. {\bf 92} (2) (1992) 285-288.


\bibitem{lee}
T. Kwen Lee, K. Shan Liu. Algebra with a finite-dimensional
maximal subalgebra, Comm. Algebra {\bf 33} (1) (2005) 339-342.


\bibitem{lucas}
T.G. Lucas, Minimal integral ring extensions, Journal of Comm
Alg. {\bf 3} (1) (2011) 47-81.


\bibitem{modica}
M. L. Modica, Maximal Subrings, Ph.D. Dissertation, (University
of Chicago, 1975).



\bibitem{abmin}
Gabriel Picavet and Martine Picavet-L'Hermitte, About minimal
morphisms, in: Multiplicative ideal theory in commutative algebra,
Springer-Verlag, New York, (2006) pp. 369-386.




\end{thebibliography}

\end{document}